\documentclass[12pt]{amsart}
\def\Z {\mathbb{Z}}

\def \co{\colon\!}

\def \ftnote{\let\thefootnote\relax\footnotetext}

\DeclareMathOperator{\terms}{terms}
\newcommand{\abs}[1]{\left|#1\right|}
\def\E{\varepsilon}
\DeclareMathOperator{\sign}{sign}

\usepackage{amsthm, amsmath, amssymb,amscd}
\usepackage{setspace, graphicx, enumerate}
\usepackage{tikz}
\usetikzlibrary{knots,hobby,decorations.markings,arrows}
\usepackage[normalem]{ulem}

\title{Minimizing intersection points of curves under virtual homotopy}
\author{Vladimir Chernov}
\address[V.~Chernov]{Department of Mathematics \\
  Dartmouth College}
\email{Vladimir.Chernov@dartmouth.edu}
\thanks{This work was partially supported by grants from the Simons Foundation
\#235674 and \#513272 to Vladimir Chernov.}

\author{David Freund}
\address[D.~Freund]{Mathematics Department \\
  Harvard University}
\email{dfreund@math.harvard.edu}

\author{Rustam Sadykov}
\address[R.~Sadykov]{Department of Mathematics \\
  Kansas State University}
\email{sadykov@ksu.edu}

\newtheorem{theorem}{Theorem}[section]
\newtheorem{lemma}[theorem]{Lemma}

\theoremstyle{remark}
\newtheorem{remark}[theorem]{Remark}

\begin{document}

\begin{abstract} 
A {\it flat virtual link} is a finite collection of oriented closed curves $\mathfrak L$ on an oriented surface $M$ considered up to {\it virtual homotopy}, i.e., a composition of elementary stabilizations, destabilizations, and homotopies. Specializing to a pair of curves $(L_1,L_2)$, we show that the minimal number of intersection points of curves in the virtual homotopy class of $(L_1, L_2)$ equals to the number of terms of a generalization of the Anderson--Mattes--Reshetikhin Poisson bracket. Furthermore, considering a single curve, we show that the minimal number of self-intersections of a curve in its virtual homotopy class can be counted by a generalization of the Cahn cobracket.
\end{abstract}
\maketitle
\leftline {\em \Small 2010 Mathematics Subject Classification. Primary: 57M99}

\leftline{\em \Small Keywords: virtual homotopy, geodesics, intersection points of curves }

\section{Virtual Homotopy of Curves}

Let $M$ be a closed oriented surface, possibly non-connected, and $\mathfrak L$ a finite collection of closed oriented 
curves on $M$. An \emph{elementary  stabilization} of  $\mathfrak L$  is a surgery on $M$ induced by cutting out 
two discs in $M$ away from $\mathfrak L$, and attaching a handle to $M$ along the resulting boundary components.  The inverse operation is called an \emph{elementary destabilization}. More precisely, let $A$ be a simple connected closed curve on $M$ in the complement to $\mathfrak L$. An \emph{elementary destabilization} of $\mathfrak L$ along $A$ consists of cutting $M$ open along $A$ and then capping the resulting boundary circles with disks. A \emph{virtual homotopy}~\cite{CahnLevi} is a composition  of elementary stabilizations, destabilizations, and homotopies. The virtual homotopy class of a collection $\mathfrak L$  is called a \emph{flat virtual link} and is denoted $\{\mathfrak L\}_V$. 

A {\em wedge} on a surface $M$ is a continuous map $S^1\vee S^1 \to M$. Analogously to the above, one defines a {\it virtual homotopy class of a collection of (ordered) wedges of circles\/} $\sqcup (S^1\vee S^1)\to M$ on a surface. Here an ordered wedge is a wedge whose components are ordered.

We are interested in counting the minimal number of intersections points between a pair of curves and the minimal number of self-intersection points for a single curve. Intersection and self-intersection points are assumed to be transverse double points. 
In what follows, we focus on curves considered up to virtual homotopy. Since virtual homotopies can affect the genus of the ambient surface, it is helpful to consider minimal genus representatives of a virtual homotopy class. Such {\it irreducible curves\/} and their uniqueness are discussed in Section~\ref{sec:irreducible}.

Recently, Cahn and the first author~\cite{CahnChernov} showed that, for a pair of curves $(L_1, L_2)$ on a surface $M$, the minimal number of intersection points of two curves homotopic to $(L_1, L_2)$ equals the number of terms in the Anderson-Mattes-Reshetikhin (AMR) Poisson bracket $[L_1, L_2]$. In Section~\ref{sec:AMR}, we prove our main theorem, showing that a similar statement is true for a generalization of the AMR bracket to the case of pairs of curves considered up to virtual homotopy.

\begin{theorem}\label{mainthm} If the virtual homotopy classes $\{L_1,L_2\}_V$ and $\{L_2,L_1\}_V$ are not equal, then the minimal number of intersection points of two curves  in the virtual homotopy class $\{L_1, L_2\}_V$ equals to the number of terms 
$\terms [L_1,L_2]_V$ in the generalization of the AMR bracket.
\end{theorem}

Cahn~\cite{Cahncobracket} showed that the number of terms of $\mu(K)$, her modification of the Turaev cobracket of a curve $K$ on a surface $M$, determines the minimal number of self-intersection points of a curve homotopic to $K$. In~\cite{Cahnvirtual}, Cahn further generalized $\mu$ to the case of flat virtual knots and conjectured that the minimal number of self-intersection points of $\{K\}_V$ would be given by an analogous formula. In Section~\ref{sec:Cahn}, we define an alternate generalization $\mu_V$ of $\mu$ to flat virtual knots and show that a similar formula holds for the case of curves considered up to virtual homotopy.

\begin{theorem} \label{mainthm2}
Let $\{K\}_V$ be a flat virtual knot and $K$ an irreducible representative of $\{K\}_V$ that is homotopic to $(K')^n$ for some primitive curve $K'$ and some $n>0$. Then the minimal number of self-intersection points of $\{K\}_V$ equals $\frac{1}{2}\terms\mu_V(\{K\}_V)+ (n-1)$.
\end{theorem}

\section{Irreducible Curves}
\label{sec:irreducible}

An elementary destabilization of a collection $\mathfrak L$ is \emph{trivial} if it chops off a sphere containing no components of $\mathfrak L$.
We say that $\mathfrak L$ is \emph{irreducible} if it admits only trivial destabilizations. Analogously, one defines the notion of an irreducible collection of (ordered) wedges of circles.


Motivated by a result of Kuperberg~\cite{GK}, Ilyutko, Manturov, and Nikonov~\cite{IMN} proved a uniqueness result for irreducible representatives of a flat virtual link. Their theorem is stated for {\it flat virtual knots} (i.e., one component links) but the proof works without change for multi-component links. Theorem~\ref{th:0.2} can also be established by a similar argument to that used by the first and third author~\cite{ChernovSadykov} to prove the uniqueness result for virtual Legendrian links.

\begin{lemma}[c.f. Theorem 1.2 in \cite{IMN}]\label{th:0.2} Every flat virtual link contains a unique (up to homotopy and an orientation preserving automorphism of $M$)  irreducible representative. The irreducible representative can be obtained from any representative by a composition of destabilizations and homotopies. 
\end{lemma}

An argument similar to the one in the proof of Lemma~\ref{th:0.2} establishes Lemma~\ref{th:0.3}.

\begin{lemma}\label{th:0.3} Every virtual homotopy class of a collection of (ordered) wedges of circles on a surface contains a unique (up to homotopy and an orientation preserving automorphism of $M$) irreducible representative. The irreducible representative can be obtained from any representative by a composition of destabilizations and homotopies. 
\end{lemma}

\begin{remark} An important consequence of Lemma~\ref{th:0.2} is that if two irreducible collections of curves $\mathfrak L_1$ and $\mathfrak L_2$ on surfaces $M_1$ and $M_2$ are virtually homotopic, then there is a homeomorphism $\varphi\co M_1\to M_2$ such that $\varphi\circ\mathfrak L_1$ is homotopic to $\mathfrak L_2$. A similar consequence holds true for a collection of (ordered) wedges of circles.  
\end{remark}

\section{The Andersen--Mattes--Reshetikhin Bracket}
\label{sec:AMR}

Let $\mathfrak L=\{L_1, L_2\}$ be a pair of curves on a surface $M.$ As usual, we assume that all the intersection points are transverse double points. Then, for any intersection point $p$, we define a wedge $\mathfrak L_p$ by taking the wedge sum of $L_1$ and $L_2$ at $p$. Define $\sign(p)$ to be $1$ if the orientation of the surface agrees with the orientation given by the tangent vectors of $L_1$ and $L_2$ at $p$, and $-1$ otherwise.

Let $W(M)$ denote the set of homotopy classes of wedges on a surface $M$, and $FW(M)$ the free abelian group generated by $W(M)$. The {\it Andersen--Mattes--Reshetikhin (AMR) bracket} is an element of $FW(M)$ given by the formal expression 
\[
    [L_1, L_2] = \sum_p \sign(p)[\mathfrak L_p],
\]
where $p$ ranges over the intersection points of $L_1$ and $L_2$.

We say that a pair $\sign(p)[\mathfrak{L}_p]$ and $\sign(q)[\mathfrak{L}_q]$ is \emph{cancelling}, if  $\mathfrak{L}_p$ is homotopic to $\mathfrak{L}_q$ and $\sign(p)=-\sign(q)$. Since $FW(M)$ is a free abelian group, the formal expression $[L_1, L_2]$ is an equivalence class where the equivalence is generated by introducing and removing cancelling pairs. A formal expression is said to be \emph{reduced} if it contains no cancelling pairs.

We generalize the AMR bracket to an operation $[L_1, L_2]_V$  in a similar way. To begin with, we trivially extend the notion of virtual homotopy of a finite collection of wedges to the notion of virtual homotopy of a finite collection of signed wedges.
Let $FW_V(M)$ denote the set of equivalence classes of formal expressions  $\sum (-1)^{\varepsilon(i)} [\mathfrak{L}_i]$ where $\mathfrak{L}_i\in W(M)$ and $\varepsilon(i)\in\{0,1\}$. A pair $(-1)^{\varepsilon(i)} [\mathfrak{L}_i]$ and $(-1)^{\varepsilon(j)} [\mathfrak{L}_j]$ is said to be \emph{virtually cancelling} if it is virtually homotopic to a cancelling pair.  The equivalence relation in $FW_V(M)$ is generated by introducing and removing cancelling pairs. In other words, $FW_V(M)$ is the quotient group of $FW(M)$ by the normal subgroup generated by the linear combinations  of the form 
\[
(-1)^{\varepsilon(i)} [\mathfrak{L}_i] - (-1)^{\varepsilon(j)} [\mathfrak{L}_j],
\] 
where the two terms $(-1)^{\varepsilon(i)} [\mathfrak{L}_i]$ and $(-1)^{\varepsilon(j)} [\mathfrak{L}_j]$  form a cancelling pair. 

A formal expression in $FW_V(M)$ is said to be \emph{reduced} if it contains no virtually cancelling pairs.

\begin{lemma}  Given a finite formal expression $\sum (-1)^{\varepsilon(i)} [\mathfrak{L}_i]$, suppose that there is a virtual homotopy of the finite collection $\{ (-1)^{\varepsilon(i)} [\mathfrak{L}_i]\}$ of signed wedges on a surface $M$ to  a collection $\{ (-1)^{\varepsilon(i)} [\mathfrak{L}'_i]\}$ such that $\sum (-1)^{\varepsilon(i)} [\mathfrak{L}'_i]=0$ in $FW(M')$ for some surface $M'$. Then  $\sum (-1)^{\varepsilon(i)} [\mathfrak{L}_i]=0$ in $FW_V(M)$.
\end{lemma}
\begin{proof}  Without loss of generality, we may assume that the  expression $\sum (-1)^{\varepsilon(i)} [\mathfrak{L}_i]$ is reduced as an element in $FW_V(M)$. Indeed, if it contained a cancelling pair, then the cancelling pair can be eliminated and thus reduce the number of terms in the formal expression. Since the number of terms in the formal expression is finite, reducing finitely many cancelling pairs results in a reduced formal expression representing the same element in $FW_V(M)$ as the initial formal expression. 

Suppose that $\sum (-1)^{\varepsilon(i)} [\mathfrak{L}_i]$ is reduced and non-trivial as an element in $FW_V(M)$. Since $\sum (-1)^{\varepsilon(i)} [\mathfrak{L}'_i]=0$ in $FW(M')$, it contains a cancelling pair, say $(-1)^{\varepsilon(s)} [\mathfrak{L}'_s]$ and $(-1)^{\varepsilon(t)} [\mathfrak{L}'_t]$. Then the formal expression $\sum (-1)^{\varepsilon(i)} [\mathfrak{L}_i]$ in $FW_V(M)$ contains a virtual cancelling pair $(-1)^{\varepsilon(s)} [\mathfrak{L}_s]$ and $(-1)^{\varepsilon(t)} [\mathfrak{L}_t]$, which contradicts the assumption that the formal expression $\sum (-1)^{\varepsilon(i)} [\mathfrak{L}_i]$ is reduced. Thus, $\sum (-1)^{\varepsilon(i)} [\mathfrak{L}_i]=0$ in $FW_V(M)$. 
\end{proof}

The virtual AMR bracket of a pair of oriented curves $L_1$ and $L_2$ on an oriented surface $M$ an element in $FW_V(M)$ represented by the formal expression
\[
    [L_1, L_2]_V = \sum_p \sign(p)[\mathfrak L_p].
\]

The proof of the following Lemma is straightforward. One checks that the value of the virtual AMR bracket under stabilization, destabilization, and the three Reidemeister moves is unchanged.

\begin{lemma} The virtual AMR bracket is well defined, i.e., its value does not depend on the choice of the representative of the virtual homotopy class of a pair of curves.
\end{lemma}

\begin{remark}
Note that the generalized AMR operation is defined on a virtual homotopy class of a pair of curves rather than on a pair of virtual homotopy classes. Furthermore, the generalized AMR operation vanishes if $\{L_1,L_2\}_V = \{L_2,L_1\}_V$ (i.e., if the component curves can be exchanged via a virtual homotopy and orientation preserving homeomorphism of the underlying surface) due to skew-symmetry. For instance, if there is a representative $(L_1,L_2)$ on a surface $M$ for which the homotopy classes of $L_1$ and $L_2$ coincide, then $[L_1,L_2]_V = [L_2,L_1]_V=0$.
\end{remark}


Let $x$ be an element of the set $FW(M)$ or $FW_V(M)$. Then there is a reduced representative of $x$, which we denote by $\bar{x}$. Suppose
\[\bar{x} = \sum_{i=1}^n c_is_i,\]
where $c_i\in \Z$, $s_i\in W(M)$, and $s_i\neq s_j$ for $i\neq j$. Then we define the {\it number of terms} of $x$ to be
\[
	\terms(x) = \sum_{i=1}^n \abs{c_i}.
	\]

\begin{lemma}\label{l:0.4} Suppose that a representative $(L_1, L_2)$ of a virtual homotopy class of a pair of curves is irreducible. Then $\terms [L_1,L_2] = \terms [L_1,L_2]_V$.
\end{lemma}
\begin{proof} 
Let $p$ be a point in the intersection $L_1\cap L_2$, and let $\mathfrak L_p$ be the wedge of $L_1$ and $L_2$ at $p$. Since every homotopy of the wedge $\mathfrak L_p$ defines a homotopy of the irreducible pair of curves $(L_1, L_2)$, we deduce that $\mathfrak L_p$ is irreducible. Consequently, the collection of wedges $\mathfrak L=\sqcup\mathfrak L_q$ on $M$, where $q$ ranges over all intersection points in $L_1\cap L_2$, is irreducible.  

By the uniqueness of the irreducible representative of a collection of wedges (see Theorem~\ref{th:0.3}),  any other irreducible representative of $\mathfrak L$ is obtained from $\mathfrak L$ by the composition of a homotopy of the collection of wedges and an orientation preserving homeomorphism of $M$ (i.e., no stabilizations or destabilizations are necessary). Under an orientation preserving homeomorphism of $M$, even though the element $[L_1, L_2]$ may change, the value $\terms [L_1, L_2]$ does not change. Consequently, any cancellation of terms in $[L_1,L_2]_V$ necessitates cancellation of the corresponding terms in $[L_1,L_2]$. Thus we conclude that $\terms [L_1, L_2]\leq \terms [L_1, L_2]_V$. Since the other inequality is obvious, we conclude that the terms are equal.
\end{proof}

\begin{remark}
When the representative $(L_1, L_2)$ of a virtual flat link is not irreducible, the value $\terms [L_1, L_2]$ may differ from the value $\terms [L_1, L_2]_V$. Indeed, there exist virtually cancelling pairs that are not cancelling. 
For example, consider the two curves on the genus three surface shown in Figure~\ref{AMRBracketExample.fig}. They intersect in two points and one can show that $\terms [L_1, L_2]=2.$
However, if one destabilizes the surface by deleting the central handle, then the two intersection points can be killed by a homotopy. Hence $\terms [L_1, L_2]_V=0$.

\begin{figure}
\includegraphics[height=1in]{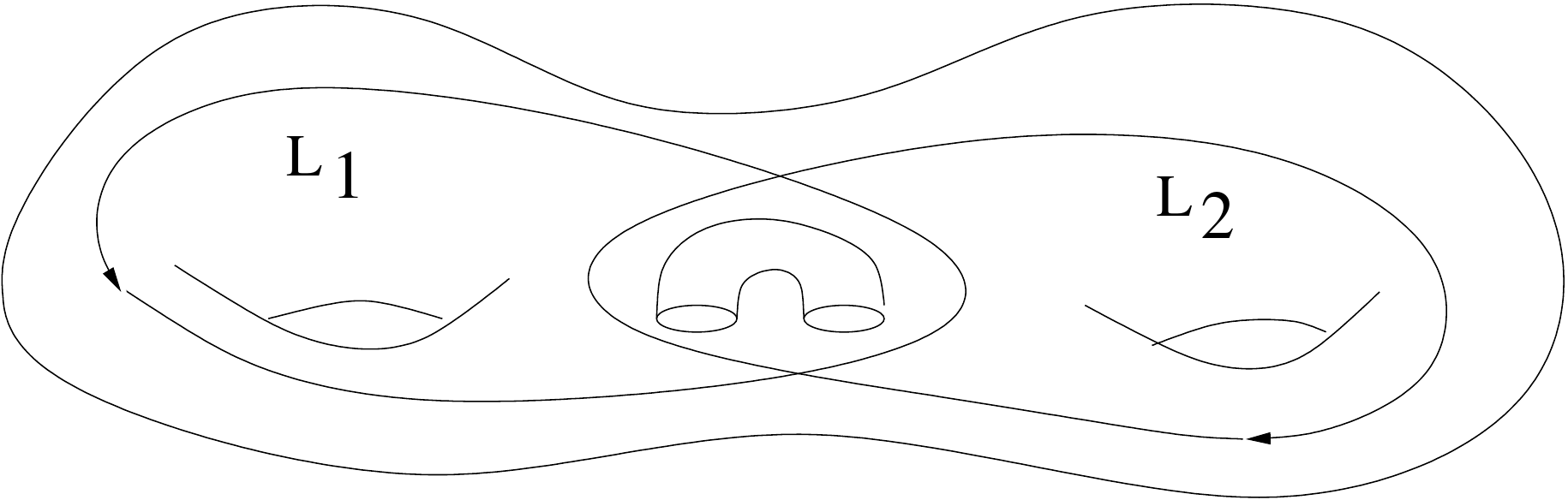}
\caption{}
\label{AMRBracketExample.fig}
\end{figure}
\end{remark}

We are now in position to prove Theorem~\ref{mainthm}.


\begin{proof} By Theorem~\ref{th:0.2}, we may assume that the representative $(L_1,L_2)$ is irreducible. Let $m$ denote the minimal number of intersection points between $L_1$ and $L_2$ in the homotopy class of $(L_1,L_2)$. By Lemma~\ref{l:0.4}, $\terms [L_1, L_2]=\terms [L_1, L_2]_V$. On the other hand, since $\{L_1,L_2\}_V\neq \{L_2,L_1\}_V$, we know that the homotopy classes of $L_1$ and $L_2$ are different and so $\terms[L_1,L_2]=m$ by~\cite{CahnChernov}.

Let $m_V$ be the minimal number of intersection points between $L_1$ and $L_2$ in the virtual homotopy class of $(L_1,L_2)$. We have $m_V\le m= \terms[L_1, L_2]_V$. On the other hand, trivially, we have $m_V\ge \terms[L_1, L_2]_V$. Thus $m_V=\terms[L_1, L_2]_V$. 
\end{proof}

\begin{remark}
Theorem~\ref{mainthm} generalizes the result of Cahn and the first author~\cite{CahnChernov} saying that, given two non-homotopic curves $(L_1, L_2)$ on a surface $M$, $\terms[L_1,L_2]$ equals the minimal number of intersection points of any pair of curves homotopic to the pair $(L_1, L_2)$.
\end{remark}

\begin{remark}\label{htpyclass}
In the proof of Theorem~\ref{mainthm}, it was sufficient to assume that the homotopy classes $[L_1]$ and $[L_2]$ are distinct for some irreducible curve $(L_1,L_2)$. 
\end{remark}

\begin{remark}\label{minimalintersectionpoints}
Assume that two curves on a surface are not homotopic to powers of a third curve. For a surface of genus greater than zero, the number of intersection points of two such curves is minimal when the curves are closed geodesics with respect to a metric of constant non-positive sectional curvature (see, for example,~\cite{TuraevViro}). In the case of two curves which are homotopic to powers of a third curve, the two closed geodesics will be powers of the same closed geodesic and one has to parallel shift one of the curves slightly to realize the minimal number of intersection points. Combining this with Theorem~\ref{mainthm}, we get that the minimal number of intersection points of a pair of curves in a given virtual homotopy class is obtained when the underlying surface is irreducible and the two curves are closed geodesics.
\end{remark}

\section{The Cahn Cobracket}
\label{sec:Cahn}

A curve $K'$ on a surface $M$ is {\it primitive} if it is not homotopic to a (nontrivial) power of another curve on $M$. Cahn~\cite{Cahncobracket} generalized the Turaev cobracket to a cobracket $\mu$ and proved that the minimal number of self-intersection points of a curve $K$ equals $\frac{1}{2}\terms\mu([K])+(n-1)$, where $K$ is homotopic to $(K')^n$ for some primitive curve $K'$ and some $n>0$.


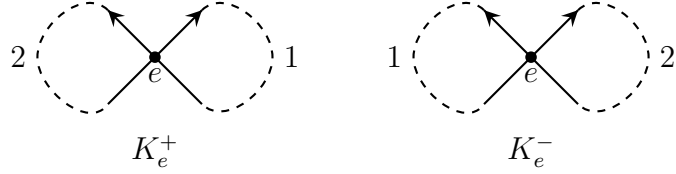
\begin{figure}
\centering
\begin{tikzpicture}[thick,scale=1.25]
\draw[draw=none, use as bounding box](0,-.5) rectangle (5,1.25);

		\draw[decoration={markings,mark=at position 1 with {\arrow[ultra thick]{stealth}}},
				postaction={decorate}
				] (0,0) -- (1,1);
		\draw[decoration={markings,mark=at position 1 with {\arrow[ultra thick]{stealth}}},
				postaction={decorate}
				] (1,0) -- (0,1);
		\draw[fill] (0.5,0.5) circle[radius=.5mm] node[below] {$e$};
				
		\draw[dashed] (0,1) to[out=135,in=90] (-.75,.5) node[left] {$2$} to[out=270,in=235] (0,0);
		\draw[dashed] (1,1) to[out=45,in=90] (1.75,.5) node[right] {$1$} to[out=270,in=-45] (1,0);

		\node at (.5,-.5) {$K_e^+$};

		\begin{scope}[xshift=4cm]
		\draw[decoration={markings,mark=at position 1 with {\arrow[ultra thick]{stealth}}},
				postaction={decorate}
				] (0,0) -- (1,1);
		\draw[decoration={markings,mark=at position 1 with {\arrow[ultra thick]{stealth}}},
				postaction={decorate}
				] (1,0) -- (0,1);
		\draw[fill] (0.5,0.5) circle[radius=.5mm] node[below] {$e$};
				
		\draw[dashed] (0,1) to[out=135,in=90] (-.75,.5) node[left] {$1$} to[out=270,in=235] (0,0);
		\draw[dashed] (1,1) to[out=45,in=90] (1.75,.5) node[right] {$2$} to[out=270,in=-45] (1,0);

		\node at (.5,-.5) {$K_e^-$};
		\end{scope}
\end{tikzpicture}
\caption{Positive labeling (left) and negative labeling (right).}
\label{fig:orderedwedges}
\end{figure}

Let $K$ be an oriented curve on a surface $M$. For a self-intersection $e$ of $K$ and $\E = \pm$, let $K_e^\E$ be the ordered wedge of circles with labeling $\E$ as depicted in Figure~\ref{fig:orderedwedges}. We denote the 
homotopy class of $K_e^\E$ by $[K_e^\E]$.

Given a self-intersection  $e$, there are two non-oriented smoothing of the curve $K$ at $e$. Of the two non-oriented smoothings of $K$, precisely one smoothed curve, say $K'$, admits an orientation compatible with the orientation of $K$. We say that $K'$ is the smoothing of $K$ according to the orientation, or, simply, the smoothing of $K$. We note that the smoothing $K'$ of $K$ has two components. 

 A self-intersection $e$ of $K$ is {\it semi-trivial} if, after smoothing $K$ at $e$, one of the resulting components intersects neither itself nor the other component. More generally, we say that an ordered wedge $K_e$ of circles with base point $e$ is semi-trivial if one of the circles intersects neither itself nor the other circle. We note that if $K$ is semi-trivial at $e$, then there is a virtual homotopy of $(K, e)$ to an oriented curve $(K', e')$ on a surface $M'$ such that after smoothing of $K'$ at $e'$ one of the components of $K'$ at $e'$ bounds a disc $D$ whose interior is disjoint from $K'$. Furthermore, there is a homotopy of $K'$ eliminating the self-intersection point $e'$. 

Let $W'(M)$ be the set of homotopy classes of ordered wedges realized on $M$, and $FW'(M)$ the free abelian group generated by $W'(M)$.  Cancelling pairs and virtual cancelling pairs for ordered wedges are defined similarly to those of unordered pairs. Let $FW'_V(M)$ denote the factor group of $FW'(M)$ by the normal subgroup generated by the terms $[K_e]$, where $K_e$ is semi-trivial, as well as by the formal expressions $(-1)^{\varepsilon(i)}[K_i]-(-1)^{\varepsilon(j)}[K_j]$ where $(-1)^{\varepsilon(i)}[K_i]$ and $(-1)^{\varepsilon(j)}[K_j]$ is a virtual cancelling pair.

For a curve $K$ on a surface $M$, the Cahn cobracket is defined as
\[\mu([K]) = \sum_e \bigl( [K_e^+] - [K_e^-]\bigr) ,\]
where $e$ ranges over all the self-intersections of $K$ for which neither of the two loops of the wedge $K_e^\varepsilon$ is null homotopic. Thus $\mu([K])$ is an element in $FW'(M)$. The formal expression for $\mu([K])$ is said to be reduced if it has no cancelling pairs. The number of terms in the reduced formal expression for $\mu([K])$ is denoted by $\terms\mu([K])$. 
Cahn showed that if $K$ is homotopic to the $n$-th power of a primitive curve, then the minimal number of self-intersection points of $[K]$ equals $\frac{1}{2}\terms\mu([K])+ (n-1)$.

We define $\mu_V(\{K\}_V)$ to be the class in $FW'_V(M)$ represented by the formal expression $\mu([K])$. As above, $\terms\mu_V(\{K\}_V)$ stands for the number of terms in the reduced formal expression for $\mu_V(\{K\}_V)$.




\begin{lemma}\label{l:4.1} Suppose $K$ is an irreducible representative of the virtual homotopy class $\{K\}_V$. Then $\terms\mu([K]) = \terms \mu_V(\{K\}_V)$.
\end{lemma}

The proof of Lemma~\ref{l:4.1} is identical to the proof of Lemma~\ref{l:0.4}, noting that $K_e^{\E}$ is irreducible for $\E=\pm$ since $K$ is irreducible. 

Similar to the proof of Theorem~\ref{mainthm}, we establish Theorem~\ref{mainthm2}.

\begin{proof} By Theorem~\ref{th:0.2}, we may assume that the representative $K$ is irreducible. Let $m$ denote the minimal number of self-intersections of $K$ in its homotopy class. Applying Lemma~\ref{l:4.1}, $\terms\mu([K])=\terms\mu_V(\{K\}_V)$. Thus
\[m = \frac{1}{2}\terms\mu([K]) + (n-1) = \frac{1}{2}\terms\mu_V(\{K\}_V) + (n-1).\]
Let $m_V$ be the minimal number of self-intersections of $K$ in $\{K\}_V$. We have
\[m_V \leq m = \frac{1}{2}\terms\mu_V(\{K\}_V) + (n-1).\]
By Proposition 5.1 in~\cite{Cahnvirtual}, $m_V \geq \frac{1}{2}\terms\mu_V(\{K\}_V) + (n-1)$. Hence we have $m_V = \frac{1}{2}\terms\mu_V(\{K\}_V) + (n-1)$.
\end{proof}


\begin{remark}\label{minimalselfintersectionpoints}
Similar to Remark~\ref{minimalintersectionpoints}, the minimal number of self-intersection points for flat virtual knots $\{K\}_V$ with primitive irreducible representatives is obtained when the representative $K$ is irreducible and a closed geodesic.
\end{remark}

\end{document}